\documentclass[12pt]{article}

\usepackage{epsfig}
\usepackage{latexsym}
\usepackage{pdfsync}

\def\Compl {{C\!\!\!\!I}}

\newtheorem{lemma}{Lemma}[section]

\newenvironment{proof}{\noindent{\bf Proof. }}{\hfill$\Box$\medskip}

\newcommand{\eq}{\begin{equation}\begin{array}{rllllllllllllllllllllllllllllllll}}
\newcommand{\ee}{\end{array}\end{equation}}
\newcommand{\bmt}{\left[ \begin{array}{ccccccccc}}
\newcommand{\emt}{\end{array}\right]}
\newcommand{\bea}{\begin{eqnarray}}

\newcommand{\eea}{\end{eqnarray}}
\newcommand{\bean}{\begin{eqnarray*}}
\newcommand{\eean}{\end{eqnarray*}}
\newcommand{\rd}{\rm\  d}

\title{Stochastic HJB Equations\\ and Regular Singular Points}
\author{Arthur J Krener
\thanks{Research supported in part by AFOSR .}
\thanks{A. J. Krener is with the Department of Applied Mathematics, Naval Postgraduate School, Monterey, CA 93943
        {\tt\small ajkrener@nps.edu}}}%

\begin{document}

\date{}

\maketitle

\begin{abstract}
In this paper we show that some HJB equations arising from both finite and infinite horizon stochastic optimal control problems 
have a regular singular point at the origin.  This makes them amenable to solution by power series techniques.   This extends the
work of Al'brecht who showed that the HJB equations of an  infinite horizon deterministic optimal control problem can have a regular singular point
at the origin,  Al'brekht  solved the HJB equations by power series, degree by degree.  In particular, we show that the infinite horizon stochastic optimal control problem
with linear dynamics, quadratic cost and bilinear noise leads to a new type of algebraic Riccati equation which we call the Stochastic 
Algebraic Riccati Equation (SARE).   If SARE can be solved then one has a complete solution   to this infinite horizon stochastic optimal control problem.
We also show that a finite horizon stochastic optimal control problem with linear dynamics, quadratic cost and bilinear noise leads to a Stochastic Differential 
Riccati Equation (SDRE) that is well known.   If these problems are  the linear-quadratic-bilinear part of a nonlinear finite horizon stochastic optimal control problem
then we show how the higher degree terms of the solutions can be computed degree by degree.  To our knowledge this computation is new.
  \end{abstract} 

\section{ Linear Quadratic Regulator with Bilinear Noise}
\setcounter{equation}{0}
Consider an infinite horizon, discounted, stochastic Linear Quadratic Regulator with Bilinear Noise (LQGB),
\bean
\min_{u(\cdot)} {1\over 2} {\rm E}\int_0^\infty e^{-\alpha t}\left(x'Qx+2x'Su+u'Ru\right) \ d t
\eean
subject to
\bean
dx&=& (Fx+Gu)\ dt+ \sum_{k=1}^r (C_{k} x+D_k u )\  dw_k\\
x(0)&=&x^0
\eean
In a previous version of this paper we studied the case with $D_k=0$ \cite{Kr18}.  

The state $x$ is $n$ dimensional, the control $u$ is $m$ dimensional and $w(t)$ is standard $r$ dimensional Brownian motion.
The matrices are sized accordingly, in particular $C_{k}$ is an $n\times n$ matrix and $D_k$ is an $n \times m$ matrix  for each $k=1,\ldots,r$.   The discount factor is $\alpha\ge 0$.

To the best of our knowledge such problems have not been considered before.  The finite horizon version of this problem can be found in Chapter 6 of the excellent treatise by Yong and Zhou \cite{YZ99}.   We will also treat  finite horizon problems in Section \ref{FH} but not in the same generality as Yong and Zhou. Throughout this note we will require that the coefficient of the noise is $O(x,u)$.   Yong and Zhou allow the coefficient to be $O(1)$ in their linear-quadratic problems.  The reason why we require $O(x,u)$ is that then the associated  stochastic Hamilton-Jacobi-Bellman equations for nonlinear extensions of 
LQGB have   regular singular points at the origin.  Hence they are amenable to solution by power series techniques.  If the noise is $O(1)$ these power series techniques have closure problems,
the equations for lower degree terms depend on higher degree terms.   If the coefficients of the noise is $O(x,u)$ then the equations can be solved degree by degree.

A first order partial differential equation with independent variable $x$ has a regular singular point at $x=0$ if the coefficients the first order partial derivatives are $O(x)$.  A second order partial differential equation has a regular singular point at $x=0$ if the coefficients the first order partial derivatives are $O(x)$ and the coefficients the second order partial derivatives are $O(x)^2$.  For more on regular singular points we refer the reader to \cite{BD09}.

If we can find a smooth scalar valued function $\pi(x)$ and a smooth $m$ vector valued $\kappa(x)$ satisfying the discounted stochastic Hamilton-Jacobi-Bellman equations (SHJB)
\bea\label{shjb1}
0&=& \mbox{min}_u \left\{ -\alpha \pi(x) +\frac{\partial \pi}{\partial x}(x) (Fx+Gu)+{1\over 2} \left( x'Qx+2x'Su+u'Ru\right)\right.
\nonumber 
\\
&&\left.  +{1\over 2}\sum_{k=1}^r (x'C'_k+u'D'_k)
 \frac{\partial^2 \pi}{\partial x^2}(x) (C_kx+D_ku) \right\} 
 \\
 \nonumber
 \kappa(x)&=& \mbox{argmin}_u  \left\{ \frac{\partial \pi}{\partial x}(x) (Fx+Gu)+{1\over 2} \left( x'Qx+2x'Su+u'Ru\right)\right.
 \\ &&\left.+{1\over 2}\sum_{k=1}^r (x'C'_k+u'D'_k)
 \frac{\partial^2 \pi}{\partial x^2}(x) (C_kx+D_ku)\right\}  
 \label{shjb2}
 \eea
then by a standard verification argument \cite{FR75} one can show that 
 $\pi(x^0)$ is the optimal cost of starting at $x^0$ and $u(0)=\kappa(x^0)$ is the optimal control at $x^0$.

We make the standard assumptions of deterministic LQR,
\begin{itemize}
\item The matrix
\bean
\bmt Q&S\\S'&R\emt
\eean
is nonnegative definite.
\item The matrix $R$ 
is positive definite.
\item The pair $F$, $G$ is stabilizable.
\item The pair $Q^{1/2}$, $F$ is detectable.
\end{itemize}

Because of the linear dynamics and quadratic cost, 
we expect that $\pi(x) $ is a quadratic function of $x$ and $\kappa(x)$ is a linear function of $x$,
\bean
\pi(x)&=& {1\over 2}x'Px\\
\kappa(x)&=& Kx
\eean
Then the stochastic Hamilton-Jacobi-Bellman equations (\ref{shjb1}, \ref{shjb2}) simplify to
\bea
0&=&-\alpha P +PF+F'P +Q -K'RK\nonumber\\
&&+\sum_{k=1}^r \left(C'_k+K'D'_k\right)P\left(C_k+D_kK\right) \label{sare}
\\
K&=&-\left(R+\sum_{k=1}^rD'_kPD_k\right)^{-1}\left(G'P+S'\right) \label{K}
\eea
We call these equations  (\ref{sare}, \ref{K}) the Stochastic Algebraic Riccati Equations (SARE).
They reduce to the  deterministic Algebraic Riccati Equations (ARE)  if $C_k=0$ and $D_k=0$.

Here is an iterative method for solving SARE.  Let $P_{(0)}$ be the solution  of the first deterministic ARE 
 \bean
0&=& -\alpha P_{(0)}+ P_{(0)}F+F'P_{(0)}+Q-(P_{(0)}G+S)R^{-1}(G'P_{(0)}+S') 
\eean
and $K_{(0)}$ be solution of the second  deterministic ARE 
\bean
K_{(0)}&=&-R^{-1}(G'P+S')
\eean

Given $P_{(\tau-1)}$  define
\bean
Q_{(\tau)}&=& Q+\sum_{k=1}^r C'_k P_{(\tau-1)}C_k\\
R_{(\tau)}&=& R+\sum_{k=1}^r D'_k P_{(\tau-1)}D_k\\
S_{(\tau)}&=& S+\sum_{k=1}^r C'_k P_{(\tau-1)}D_k
\eean

Let
 $P_{(\tau)}$ be the solution  of
\bean
0&=& -\alpha P_{(\tau)} +P_{(\tau)}F+F'P_{(\tau)}+Q_{(\tau)}-(P_{(\tau)}G+S_{(\tau)})R_{(\tau)}^{-1}(G'P_{(\tau)}+S'_{(\tau)})
\eean
and
\bean
K_{(\tau)}&=&-R_{(\tau)}^{-1}\left(G'P_{(\tau)}+S_{(\tau)}'\right) 
\eean

If the iteration on $P_{(\tau)}$  nearly converges, that is, for some $\tau$, $P_{(\tau)}\approx P_{(\tau-1)}$ then $P_{(\tau)}$ and $ K_{(\tau)}$ are approximate solutions to SARE

The solution $P$ of the deterministic ARE is the kernel of the optimal cost of a deterministic LQR and since   
\bean
\bmt Q& S\\ S'& R\emt \le \bmt Q_{(\tau-1)}& S_{(\tau-1)}\\S'_{(\tau-1)}& R_{(\tau-1)}\emt \le \bmt Q_{(\tau)}& S_{(\tau)}\\S'_{(\tau)}& R_{(\tau)}\emt
\eean
it follows that $P_{(0)}\le P_{(\tau-1)} \le P_{(\tau)} $,  the iteration is monotonically increasing.
  We have found computationally that  if matrices $C_k$ and $D_k$ are not too big then the iteration conveges.  But if the $C_k$ and $D_k$ are about the same size as $F$ and $G$ or larger the iteration can diverge.    Further study of this issue is needed.   The iteration does converge in the following simple example.

\section{LQGB Example}
\setcounter{equation}{0}
Here is a simple example with $n=2,m=1,r=2$.
\bean
\min_u {1\over 2}\int_0^\infty \|x\|^2+u^2\ dt
\eean
subject to 
\bean
dx_1&=& x_2\ dt+0.1 x_1 \ dw_1\\
dx_2&=&u\ dt+0.1 (x_2 +u)\ dw_2
\eean
In other words 
\bean
Q=\bmt 1&0\\0&1\emt,& S=\bmt 0\\1\emt, &R=1\\
F= \bmt 0&1\\0&0\emt,& G=\bmt 0\\1\emt& \\
C_1=\bmt 0.1&0\\0&0\emt, &C_2=\bmt 0&0\\0&0.1\emt&\\
D_1=\bmt 0\\0\emt, &D_2=\bmt 0\\0.1\emt&
\eean

The solution of the noiseless ARE is
\bean
P&=& \bmt 1.7321 & 1.000\\1.000&1.7321\emt
\\
K&=&-\bmt 1.0000& 1.7321\emt
\eean
The eigenvalues of the noiseless  closed loop matrix $F+GK$ are $-0.8660\pm0.5000i $.

The above iteration converges to the solution of the noisy SARE in eight iterations, the solution is
\bean
P&=& \bmt 1.7625 & 1.0176\\1.0176&1.7524\emt\\
\\
K&=&-\bmt 1.0176&1.7524\emt
\eean
The eigenvalues of the noisy closed loop matrix $F+GK$ are $-0.8762\pm 0.4999i$.

As expected the noisy system is more difficult to control than the noiseless system.  It should be noted that the above iteration diverged to infinity
when the noise coefficients were increased from $0.1$ to $1$.

\section{Nonlinear Infinite Horizon HJB}
Suppose the problem is not linear-quadratic,  the dynamics is given by an Ito equation
\bean
dx&=& f(x,u) \ dt +\sum_{k=1}^r\gamma_k(x,u) \ dw_k
\eean
 and the criterion to be minimized is
 \bean
\min_{u(\cdot)} {\rm E}\int_0^\infty e^{-\alpha t}l(x,u)\ d t
\eean

We assume that $f(x,u), \gamma_k(x,u), l(x,u) $ are smooth functions that have Taylor polynomial expansions
around $x=0,u=0$,
\bean
f(x,u)&=& Fx+Gu+f^{[2]}(x,u)+\ldots+f^{[d]}(x,u)+O(x,u)^{d+1}\\
\gamma_k(x,u)&=& C_kx+D_ku+\gamma_k^{[2]}(x,u)+\ldots+\gamma_k^{[d]}(x,u)+O(x)^{d+1}\\
l(x,u)&=&{1\over 2}\left(x'Qx+2x'Su+u'Ru\right) +l^{[3]}(x,u)+\ldots+l^{[d+1]}(x,u)+O(x,u)^{d+2}
\eean
where $^{[d]}$ indicates the homogeneous polynomial terms of degree $d$.

  Then the the discounted stochastic Hamilton-Jacobi-Bellman equations become
  \bea
0&=& \mbox{min}_u \left\{ -\alpha \pi(x) +\frac{\partial \pi}{\partial x}(x) f(x,u)+l(x,u)\right. \label{shjb3}
\nonumber 
\\
&&\left.  +{1\over 2}\sum_{k=1}^r \gamma'_k(x,u)
 \frac{\partial^2 \pi}{\partial x^2}(x) \gamma_k(x,u)\right\} 
 \\
 \kappa(x)&=& \mbox{argmin}_u    \left\{ -\alpha \pi(x) +\frac{\partial \pi}{\partial x}(x) f(x,u)+l(x,u)\right. \nonumber
\\   \label{shjb4}
&&\left.  +{1\over 2}\sum_{k=1}^r \gamma'_k(x,u)
 \frac{\partial^2 \pi}{\partial x^2}(x) \gamma_k(x,u)\right\} 
 \eea
 
 If the control enters the dynamics affinely,
\bean
f(x,u)&=& f^0(x) +f^u(x)u\\
\gamma_k(x,u)&=&\gamma^0_k(x)+\gamma^u_k(x)u
\eean
 and $l(x,u)$ is always strictly convex in $u$ for every $x$ then the quantity to be minimized in (\ref{shjb3}) is strictly convex in $u$.

If we assume that (\ref{shjb3})
 is strictly convex in $u$  then the HJB  equations  (\ref{shjb3}, \ref{shjb4}) simplify to
\bea
0&=&-\alpha \pi(x) +\frac{\partial \pi}{\partial x}(x) f(x,\kappa(x))+l(x,\kappa(x))\label{shjb5} \\
&& \nonumber+{1\over 2}\sum_{k=1}^r \gamma'_k(x,\kappa(x))
 \frac{\partial^2 \pi}{\partial x^2}(x) \gamma_k(x,\kappa(x)) 
 \\
 0&=&\frac{\partial \pi}{\partial x}(x) \frac{\partial f}{\partial u}(x,\kappa(x))+\frac{\partial l}{\partial u}(x,\kappa(x))
 \label{shjb6}
 \\&& +\sum_{k=1}^r \gamma'_k(x,\kappa(x))
 \frac{\partial^2 \pi}{\partial x^2}(x) \frac{\partial \gamma_k}{\partial u}(x,\kappa(x)) \nonumber
\eea

Because $f(x,u)=O(x,u)$ and $\gamma_k(x,u)=O(x,u)$, (\ref{shjb5}) has a regular singular point at $x=0,u=0$ and so is 
amenable to power series solution techniques. If $\gamma_k(x,u)=O(1)$ then there is persistent noise
that must be overcome by persistent control action.  Presumably then the optimal cost is infinite.

  Following Al'brekht \cite{Al61} we assume that the optimal cost $\pi(x)$ and the optimal
feedback have Taylor polynomial expansions
\bean
\pi(x)&=& {1\over 2}x'Px +\pi^{[3]}(x)+\ldots+\pi^{[d+1]}(x)+O(x)^{d+2}\\
\kappa(x)&=&Kx+\kappa^{[2]}(x)+\ldots+\kappa^{[d]}(x)+O(x)^{d+1}
\eean
We  plug all these expansions into the simplified SHJB equations (\ref{shjb5}, \ref{shjb6}). 
At lowest degrees, degree two in (\ref{shjb5}) and degree one in (\ref{shjb6}) we get the familiar SARE
 (\ref{sare}, за\ref{K}).

If  (\ref{sare}, \ref{K}) are solvable 
then  we may proceed to the next degrees, degree three in (\ref{shjb5})  and degree two in (\ref{shjb6}).  
\bea
0&=&\frac{\partial \pi^{[3]}}{\partial x}(x) (F+GK)x+x'Pf^{[2]}(x,Kx)+l^{[3]}(x,Kx) \label{shjb7}\\&&
+{1\over 2} \sum_k x'(C'_k +K'D'_k) \frac{\partial^2 \pi^{[3]}}{\partial x^2}(x) (C_k+D_kK) x \nonumber
\\&&+\sum_k x'(C'_k +K'D_k)P\gamma_k^{[2]}(x,Kx)\nonumber \\  \nonumber
\\
0&=& \frac{\partial \pi^{[3]}}{\partial x}(x) G +x'P\frac{\partial f^{[2]}}{\partial u}(x,Kx)+\frac{\partial l^{[3]}}{\partial u}(x,Kx)
 \label{shjb8}
 \\&&+\sum_k x'(C_k+D_kK)'\left(P\frac{\partial \gamma^{[2]}_k}{\partial u}(x,Kx)+ \frac{\partial^2 \pi^{[3]}}{\partial x^2}(x) D_k\right) \nonumber
\\&& +\sum_k \gamma^{[2]}_k(x,Kx)PD_k \nonumber 
 +(\kappa^{[2]}(x))'\left(R+\sum_k D'_kPD_k\right)  \nonumber
 \eea
Notice the first equation (\ref{shjb7})  is a square linear equation for the  unknown $\pi^{[3]}(x)$,
the other unknown $\kappa^{[2]}(x)$ does not appear in it.
 If we can solve the first equation (\ref{shjb7})  for $\pi^{[3]}(x)$
 then we can solve the second equation (\ref{shjb8})  for $\kappa^{[2]}(x)$ because of the standard assumption that $R$ is invertible so $R+\sum_k D_kPD_k$
 must also be invertible.
 
 In the deterministic case the eigenvalues of the linear operator
 \bea \label{dop}
 \pi^{[3]}(x) &\mapsto& \frac{\partial \pi^{[3]}}{\partial x}(x) (F+GK)x
 \eea 
 are the sums of three eigenvalues of $F+GK$.  Under the standard LQR assumptions all the  eigenvalues of $F+GK$ are in the open left half plane so any
 sum of three eigenvalues of $F+GK$ is different from zero and the operator (\ref{dop}) is invertible.

  In the stochastic case the relevant linear operator is a sum of two operators
 \bea \label{sop}
 \pi^{[3]}(x) &\mapsto& \frac{\partial \pi^{[3]}}{\partial x}(x) (F+GK)x \\&&
 +{1\over 2} \sum_k x'(C'_k +K'D'_k) \frac{\partial^2 \pi^{[3]}}{\partial x^2}(x) (C_k+D_kK) x \nonumber
 \eea 
 
 Consider a simple version of the second  operator, for some $C$,
  \bea \label{sim} \pi^{[3]}(x) &\mapsto&{1\over 2} x'C'\frac{\partial^2 \pi^{[3]}}{\partial x^2}(x)Cx
 \eea
 Suppose $C$ has a complete set of left eigenpairs, $\lambda_i\in \Compl,\  w^i\in \Compl^{1\times n}$ for $i=1,\ldots,n$,
 \bean
 w^i C&=& \lambda_i w^i
 \eean
 Then the eigenvalues of (\ref{sim}) are of the form
$
 \lambda_{i_1}\lambda_{i_2}+\lambda_{i_2}\lambda_{i_3}+ \lambda_{i_3}\lambda_{i_1}
$
and the corresponding eigenvectors are
 $(w^{i_1}x)(w^{i_2}x)(w^{i_3}x)$
 for  for $ 1\le i_1\le  i_2\le i_3$. 
 But this analysis does not completely clarify whether the operator (\ref{sop}) is invertible.  Here is one case where it is known to be invertible.
 
 Consider the space of cubic polynomials $\pi(x)$.   We can norm this space using the standard $L_2$ norm on the vector of coefficients of $\pi(x)$ which we denote by $\|\pi(x)\|$.   Then there is an induced  norm on operators like (\ref{dop}), (\ref{sop}) and
 \bea  \label{pop}
  \pi^{[3]}(x) &\mapsto& {1\over 2} \sum_k x'(C'_k +K'D_k) \frac{\partial^2 \pi^{[3]}}{\partial x^2}(x) (C_k+D_kK) x \nonumber
 \eea 
Since the operator (\ref{dop}) is invertible its inverse has an operator norm  $\rho<\infty$.   If all the eigenvalues of $F+GK $ have real parts less that $-\tau$ then 
${ 1 \over \rho} \ge 3\tau$.
Let $\sigma$ be the supremum operator norms of $C_k+D_kK$ for $k=1.\ldots, r$.   Then from the discussion above we know that the operator norm of (\ref{pop}) is bounded above by ${3r\sigma^2\over 2}$

 \begin{lemma}
If $\tau > {r\sigma^2 \over 2}$ then
 the operator (\ref{sop}) is invertible.
 \end{lemma}
 \begin{proof}
 Suppose (\ref{sop}) is not invertible then there exist a cubic polynomial $\pi(x)\ne 0$ such that
 \bean
 \frac{\partial \pi^{[3]}}{\partial x}(x) (F+GK)x &=&
 -{1\over 2} \sum_k x'(C'_k +K'D_k) \frac{\partial^2 \pi^{[3]}}{\partial x^2}(x) (C_k+D_kK) x \eean
 so 
 \bean
  \left\| \frac{\partial \pi^{[3]}}{\partial x}(x) (F+GK)x \right\|=\left\| {1\over 2} \sum_k x'(C'_k +K'D_k) \frac{\partial^2 \pi^{[3]}}{\partial x^2}(x) (C_k+D_kK) x \right\|
 \eean
 But we know that 
 \bean
 \left\| \frac{\partial \pi^{[3]}}{\partial x}(x) (F+GK)x \right\|\ge{1\over \rho}\|\pi(x)\|\ge 3\tau \|\pi(x)\|>{3r\sigma^2 \over 2}\|\pi(x)\|
 \eean
 while
 \bean
 \left\| {1\over 2} \sum_k x'(C'_k +K'D_k) \frac{\partial^2 \pi^{[3]}}{\partial x^2}(x) (C_k+D_kK) x \right\|\le{3r\sigma^2 \over 2}\|\pi(x)\|
 \eean
  \end{proof}
 
 The takeaway message from this lemma is that if the nonzero entries of $C_k, D_k$ are small relative to the nonzero entries of $F, G$ then  we can expect that (\ref{sop})  will be invertible.
 
 There are two ways to try solve (\ref{shjb7}), the iterative approach or the direct approach .   
 We have written Matlab software to solve  the deterministic version of these equations.  This suggests an iteration scheme similar to the above for solving SARE. Let $\pi^{[3]}_{(0)}$
 be the solution of the deteministic version of (\ref{shjb7}) where $C_k=0,\ D_k=0$.
Given $\pi^{[3]}_{(\tau-1)}(x)$  define 
\bean
l^{[3]}_{(\tau)}(x,u)&=&l^{[3]}(x,u)+{1\over 2} \sum_k x'(C'_k +K'D_k) \frac{\partial^2 \pi^{[3]}_{\tau-1}}{\partial x^2}(x) (C_k+D_kK) x \nonumber
\\&&+\sum_k x'(C'_k +K'D_k)P\gamma_k^{[2]}(x,u)\nonumber
\eean
 and let $\pi^{[3]}_{(\tau)}$ be the solution of 
 \bean
0&=&\frac{\partial \pi^{[3]}_{(\tau)}}{\partial x}(x) (F+GK)x+x'Pf^{[2]}(x,Kx)+l^{[3]}_{(\tau)}(x,Kx) 
 \eean 
  If this iteration converges then we have the solution to (\ref{shjb7}).
  We have also written Matlab software to solve (\ref{shjb7}) directly assuming the operator (\ref{sop}) is invertible.  

 If (\ref{shjb7}) is solvable then solving (\ref{shjb8})  for $\kappa^{[2]}(x)$ is straightforward
as we have assumed that $R$ is invertible.
 If these equations are solvable then we can move on to the equations for $\pi^{[4]}(x)$ and $\kappa^{[3]}(x)$ and higher degrees.

 It should be noted that if the Lagrangian is an even function  and the dynamics is an odd function then the optimal cost $\pi(x)$ is 
  an even function and the optimal feedback $\kappa(x)$ is an odd function.

\section{Nonlinear Example}
\setcounter{equation}{0}
Here is a simple example with $n=2,m=1,r=1$.  Consider a pendulum of length $1\ m$ and mass $1\ kg$ orbiting approximately 400 kilometers
above Earth on the International Space Station (ISS).   The "gravity constant" at this height is approximately $g=8.7\ m/sec^2$.  The pendulum can be controlled
by a torque $u$ that can be applied at the pivot and there is damping at the pivot with  linear damping constant $c=0.1\ kg/sec$ and  cubic damping constant $c_3= 0.05\ kg\ sec/m^2$.   Let $x_1$ denote the angle of pendulum measured counter clockwise from the outward pointing ray from the center of the Earth and let $x_2$ denote the angular velocity.  The determistic equations of motion are
\bean
\dot{x}_1&=& x_2
\\
\dot{x}_2&=& lg\sin x_1 -c_1 x_2-c_3 x_2^3 +u
\eean
But the shape of the earth is not a perfect sphere and its density  is not uniform  so there are fluctuations in the "gravity constant".  We set these fluctuations in the "gravity constant" at one percent although they are probably smaller.  There might also be fluctuations in the damping constants of around one percent.   Further assume that
the commanded torque is not always realized and the relative error in the actual torque  fluctuates around one percent.  
We model these stochastically by three  white noises 
\bean
dx_1&=& x_2\ dt\\
dx_2&=&\left(lg\sin x_1 -c_1x_2-c_3x_2^3+u\right)\ dt\\
&&+0.01 lg\sin x_1 \ dw_1- 0.01(c_1 x_2+c_3x_2^3)\ dw_2 +0.01u\ dw_3
\eean

This is  an example about how stochastic models with noise coefficients of order $O(x)$ can arise.   If the noise is modeling an uncertain  environment then its coefficients are likely  to be  $O(1)$.  But if it is  the  model  that is  uncetain then noise coefficients are likely to be $O(x)$. 

The goal is to find a feedback $u=\kappa(x)$ that stabilizes the pendulum to straight up in spite of the noises so we take the criterion to be 
\bean
\min_u {1\over 2}\int_0^\infty \|x\|^2+u^2\ dt
\eean
with discount factor is $\alpha=0$.

Then 
\bean
F=\bmt 0&1\\8.7&0.1\emt,& G=\bmt 0\\1\emt,& \\
Q=\bmt 1&0\\0&1\emt, & R=1,& S=\bmt 0\\0\emt\\
C_1=\bmt 0&0\\0.087&0\emt,&C_2=\bmt 0&0\\0&-0.001\emt,& C_3=\bmt 0&0\\0&0\emt\\
D_1=\bmt 0\\0\emt,& D_2 =\bmt 0\\0\emt,& D_3 =\bmt 0\\0.01\emt
\eean

Because the Lagrangian is an even function  and the dynamics is an odd function of $x,u$, we know that
$\pi(x)$ is an even function  of $x$ and $\kappa(x)$ s an odd function of $x$.   

We have computed the optimal cost $\pi(x)$ to degree $6$ and the optimal feedback $\kappa(x)$ to degree $5$,
\bean
\pi(x)&=&26.7042x_1^2+   17.4701x_1x_2    2.9488x_2^2\\&&
-4.6153x_1^4   -2.9012x_1^3x_2   -0.5535x_1^2x_2^2   -0.0802 x_1x_2^3  -0.0157x_2^4\\
&&
0.3361x_1^6+    0.1468x_1^5x_2   -0.0015x_1^4x_2^2   -0.0077x_1^3x_2^3 \\
&&  -0.0022x_1^2x_2^4   -0.0003x_1x_2^5    +0.0000x_2^6
\\
\kappa(x)&=&-17.4598x_1   -5.8941x_2 \\
&& +2.9012x_1^3+    1.1071x_1^2x_2+    0.2405x_1x_2^2+    0.0628x_2^3\\
&&
-0.1468x_1^5+   0.0031x_1^4x_2+    0.0232x_1^3x_2^2\\&&+    0.0089x_1^2x_2^3+    0.0014x_1x_2^4   -0.0002x_2^5
\eean

In making this computation we are approximating $\sin x_1$ by its Taylor polynomials 
\bean
\sin x_1&=& x_1-{x_1^3\over 6} +{x_1^5 \over 120}+\ldots
\eean
The alternating signs of the odd terms in these polynomials are reflected in the nearly  alternating signs in the Taylor polynomials of the optimal cost $\pi(x)$
and optimal feedback $\kappa(x)$.  If we take a first degree approximation to $\sin x_1$ we are overestimating the gravitational force 
pulling the pendulum from its upright position pointing so $\pi^{[2}(x)$ overestimates the optimal cost
and the feedback $u=\kappa^{[1]}(x)$ is stronger than it needs to be.  The latter could be a problem if there is a bound on the magnitude of $u$ that we ignored in the analysis.
If we take a third degree approximation to $\sin x_1$ then $\pi^{[2]}(x)+\pi^{[4]}(x)$ underestimates the optimal cost
and the feedback $u=\kappa^{[1]}(x)+\kappa^{[3]}(x)$ is weaker than it needs to be.
If we take a fifth degree approximation to $\sin x_1$ then $\pi^{[2]}(x)+\pi^{[4]}(x)+\pi^{[6]}(x)$ overestimates the optimal cost but by a smaller margin than 
$\pi^{[2}(x)$.  The feedback     $u=\kappa^{[1]}(x)+\kappa^{[3]}(x)+\kappa^{[5]}(x)$  is stronger than it needs to be
but by a smaller margin than $u=\kappa^{[1]}(x)$.

\section{Finite Horizon Stochastic Nonlinear Optimal Control  Problem} \label{FH}
\setcounter{equation}{0}
Consider the finite horizon stochastic nonlinear optimal control  problem,
\bean
\min_{u(\cdot)}  {\rm E}\left\{ \int_0^T l(t,x,u) \rd t+\pi_T(x(T))\right\}
\eean
subject to 
\bean
d x&=& f(t,x,u)dt+\sum_{k=1}^r\gamma_k(t,x,u)d w_k\\
x(0)&=&x^0
\eean
Again we assume that $ f, l,\gamma_k, \pi_T$ are sufficiently smooth.

If they exist and are smooth the optimal cost $\pi(t, x) $ of starting at $x$ at time $t$ and the optimal feedback
$u(t)=\kappa(t,x(t))$ satisfy the time dependent  Hamilton-Jacobi-Bellman equations (HJB) 
\bean
0&=& \mbox{min}_u \left\{ \frac{\partial \pi}{\partial t}(t,x) +\frac{\partial \pi}{\partial x}(t,x) f(t,x,u) 
+l(t,x,u)\right. \\&&\left. +{1\over 2}\sum_{l=1}^k \gamma'_k(t,x,u)
 \frac{\partial^2 \pi}{\partial x^2}(t,x) \gamma_k(t,x,u) \right\}
 \\
 0&=& \mbox{argmin}_u \left\{ \sum_i\frac{\partial \pi}{\partial x_i}(t,x) f_i(t,x,u) +l(t,x,u)\right. 
 \\&&\left. +{1\over 2}\sum_{l=1}^k \gamma'_k(t,x,u)
 \frac{\partial^2 \pi}{\partial x^2}(t,x) \gamma_k(t,x,u) \right\}
  \eean

If the quantity to be minimized
is strictly convex in $u$ then HJB equations simplify to
\bea \nonumber
0&=&  \frac{\partial \pi}{\partial t}(t,x) + \sum_i\frac{\partial \pi}{\partial x_i}(t,x) f_i(t,x,\kappa(x)) +l(t,x,\kappa(x)) 
\\&&+{1\over 2} \sum_{k=1}^r
\gamma'_k(t,x,\kappa(x)) \frac{\partial^2 \pi}{\partial x^2}(t,x) \gamma_k(t,x,\kappa(x)) \label{hjb1t}\\
\nonumber
 \\
 \label{hjb2t}
 0&=& \sum_{i,k} \frac{\partial \pi}{\partial x_i}(x)  \frac{\partial f _i}{\partial u_k}(t,x,\kappa(x)) +\sum_k  \frac{\partial l }{\partial u_k}(t,x,\kappa(x)) 
 \\&& +\sum_{k=1}^r \gamma'_k(t,x,\kappa(x))
 \frac{\partial^2 \pi}{\partial x^2}(x) \frac{\partial \gamma_k}{\partial u}(t,x,\kappa(x)) \nonumber
\eea
These equations are integrated backward in time from the
 final condition
\bea \label{hjbT}
\pi(T,x)&=& \pi_T(x)
\eea

Again we assume that we have the  following Taylor expansions
\bean
f(t,x,u)&=& F(t)x+G(t)u+f^{[2]}(t,x,u)+f^{[3]}(t,x,u)+\ldots\\
l(t,x,u)&=& {1\over 2}\left( x'Q(t)x+u'R(t)u\right)+l^{[3]}(t,x,u)+l^{[4]}(t,x,u)+\ldots\\
\gamma_k(t,x)&=& C_k(t)x+\gamma_k^{[2]}(t,x)+\beta_{k}^{[3]}(t,x)+\ldots\\
\pi_T(x)&=& {1\over 2} x'P_Tx+\pi_T^{[3]}(x)+\pi_T^{[4]}(x)+\ldots\\
\pi(t,x)&=& {1\over 2} x'P(t)x+\pi^{[3]}(t,x)+\pi^{[4]}(t,x)+\ldots\\
\kappa(t,x)&=& K(t)x+\kappa^{[2]}(t,x)+\kappa^{[3]}(t,x)+\ldots
\eean
where $^{[r]}$ indicates terms of homogeneous degree $r$ in $x,u$ with coefficients that are continuous  functions of $t$.

The key assumption is that $\gamma_k(t,0)=0$
for then (\ref{hjb1t}) has a regular singular point at $x=0$ and so is amenable to power series methods.

We plug these expansions into the simplified time dependent HJB equations and collect terms of lowest degree, that is, degree  two in (\ref{hjb1t}), degree one in (\ref{hjb2t}) and degree two in (\ref{hjbT}).
\bean
0&=& \dot{P}(t)+P(t)F(t)+F'(t)P(t)+Q(t)-K'(t)R(t)K(t)\\
&&
 +\sum_k \left(C'_k(t)+K'(t)D'_k(t)\right)P(t)\left(C_k(t)+ D_k(t)K(t)\right) \\
K(t)&=& -\left(R(t)+\sum_{k=1}^rD'_k(t)P(t)D_k(t)\right)^{-1} (G'(t) P(t)+S(t))\\
P(T)&=& P_T
\eean
We call these equations the stochastic differential Riccati equation (SDRE).     Similar equations in more generality can be found in \cite{YZ99} but since we are interested in nonlinear problems 
 we require that $\gamma_k(t,x)=O(x)$  so that the stochastic HJB equations have a regular singular at the origin.

If SDRE are solvable we may proceed to the next degrees, degree three in (\ref{hjb1t}),  and degree two  in (\ref{hjbT}).
\bean
0&=&\frac{\partial  \pi^{[3]}}{\partial t}(t,x)+ \frac{\partial \pi^{[3]}}{\partial x}(t,x) (F(t)+G(t)K(t))x\\
&&+x'P(t)f^{[2]}(t,x,K(t)x)+l^{[3]}(t,x,Kx)\\&&
\\&&
 +{1\over 2}\sum_k x'C'_k(t)  \frac{\partial^2 \pi^{[3]}}{\partial x^2}(t,x) \left(C_k+D_k(t)K(t)\right)(t) x\\
 &&
 +\sum_k x'\left(C'_k(t)+K'(t)D'_k(t)\right)P(t)\gamma_k^{[2]}(t,x)\\
\\
0&=& \frac{\partial \pi^{[3]}}{\partial x}(t,x) G(t) +x'P(t)\frac{\partial f^{[2]}}{\partial u}(t,x,K(t)x)+\frac{\partial l^{[3]}}{\partial u}(t,x,K(t)x)
 \\&+&\sum_k x'(C_k(t)+D_k(t)K(t))'\left(P(t)\frac{\partial \gamma^{[2]}_k}{\partial u}(x,K(t)x)+ \frac{\partial^2 \pi^{[3]}}{\partial x^2}(x) D_k(t)\right) \nonumber
\\&+& \sum_k \gamma^{[2]}_k(x,K(t)x)P(t)D_k(t)
 +(\kappa^{[2]}(t,x))'\left(R(t)+\sum_k D'_k(t)PD_k(t)\right) \\ \nonumber
 \eean

Notice again the unknown $\kappa^{[2]}(t,x)$ does not appear in the first equation which is linear ode for  
$ \pi^{[3]}(t,x)$ running backward in time from the terminal condition,
\bean
\pi^{[3]}(t,x)&=& \pi^{[3]}_T(x)
\eean  
After we have  solved it then the second equation for $\kappa^{[2]}(t,x)$ is easily solved  because of the standard assumption that $R(t)$ is invertible and hence $R(t)+\sum_k D'_k(t)PD_k(t)$ is invertible.  
 
 The higher degree terms can be found in a similar fashion.

\section{Acknowledgements}
The author would like to thank Mark Davis, Wendell Fleming, Alan Laub, George Yin and especially Peter Caines for their helpful comments.


\begin{thebibliography}{BP67}
  
  \bibitem{Al61}
E.~G.~Al'brekht,
 {\it On the optimal stabilization of nonlinear systems},
PMM-J. Appl. Math. Mech., 25:1254-1266, 1961.

\bibitem{BD09}
W.~E.~Boyce and R.~C.~DiPrima, {\it Elementary Differential Equations and Boundary Value Problems}, Tenth Edition,
Wiley, New Jersey, 2009.


\bibitem{FR75}
W.~Fleming and R.~Rishel,  {\it Deterministic and Stochastic Optimal Control},
Springer. New York,  1975.

\bibitem{Kr18}
A.~J.~Krener, {Stochastic HJB Equations and Regular Singular Points},
arXiv:1806.04120v1 [math.OC] 11 Jun 2018.
   
   \bibitem{YZ99}
   J.~Yong and X.~J.~Zhou, {\it Stochastic Controls, Hamiltonian Systems and HJB Equations}, Springer. New York, 1999.
   
\end{thebibliography}
\end{document}